\documentclass[a4paper,11pt]{article}
\usepackage[utf8]{inputenc}
\usepackage{a4wide}
\usepackage{latexsym,amsfonts,amsmath,amssymb,mathrsfs,url,amsthm}
\usepackage{mathtools}
\usepackage{color,graphicx}
\usepackage{lipsum}
\usepackage{xcolor}
\providecommand{\keywords}[1]
{
  \noindent \small	
  \textbf{Keywords:} #1
}
\providecommand{\amscode}[1]
{
  \noindent \small	
  \textbf{AMS subject classifications:} #1
}
\newtheorem{theorem}{Theorem}
\newtheorem{lemma}[theorem]{Lemma}
\newtheorem{corollary}[theorem]{Corollary}

\newcommand{\rd}{\, \mathrm{d}}

\newcommand{\bsk}{\boldsymbol{k}}
\newcommand{\bsell}{\boldsymbol{\ell}}
\newcommand{\bszero}{\boldsymbol{0}}

\newcommand{\bsx}{\boldsymbol{x}}
\newcommand{\wor}{\mathrm{wor}}
\newcommand{\CC}{\mathbb{C}}
\newcommand{\II}{\mathbb{I}}
\newcommand{\NN}{\mathbb{N}}
\newcommand{\PP}{\mathbb{P}}
\newcommand{\RR}{\mathbb{R}}
\newcommand{\ZZ}{\mathbb{Z}}
\newcommand{\Acal}{\mathcal{A}}
\DeclareMathOperator{\supp}{supp}
\DeclareMathOperator{\width}{width}

\allowdisplaybreaks

\title{Strong tractability for multivariate integration in a subspace of the Wiener algebra\thanks{The work of the author is supported by JSPS KAKENHI Grant Number 23K03210.}}
\author{Takashi Goda\thanks{School of Engineering, University of Tokyo, 7-3-1 Hongo, Bunkyo-ku, Tokyo 113-8656, Japan ({\tt goda@frcer.t.u-tokyo.ac.jp})}
}
\date{\today}

\begin{document}

\maketitle
\begin{abstract}
Building upon recent work by the author, we prove that multivariate integration in the following subspace of the Wiener algebra over $[0,1)^d$ is strongly polynomially tractable:
\[ F_d:=\left\{ f\in C([0,1)^d)\:\middle| \: \|f\|:=\sum_{\bsk\in \ZZ^{d}}|\hat{f}(\bsk)|\max\left(\width(\supp(\bsk)),\min_{j\in \supp(\bsk)}\log |k_j|\right)<\infty \right\},\]
with $\hat{f}(\bsk)$ being the $\bsk$-th Fourier coefficient of $f$, $\supp(\bsk):=\{j\in \{1,\ldots,d\}\mid k_j\neq 0\}$, and $\width: 2^{\{1,\ldots,d\}}\to \{1,\ldots,d\}$ being defined by
\[ \width(u):=\max_{j\in u}j-\min_{j\in u}j+1,\]
for non-empty subset $u\subseteq \{1,\ldots,d\}$ and $\width(\emptyset):=1$. Strong polynomial tractability is achieved by an explicit quasi-Monte Carlo rule using a multiset union of Korobov's $p$-sets. We also show that, if we replace $\width(\supp(\bsk))$ with 1 for all $\bsk\in \ZZ^d$ in the above definition of norm, multivariate integration is polynomially tractable but not strongly polynomially tractable. 
\end{abstract}

\amscode{41A55, 41A58, 42B05, 65D30, 65D32}

\section{Introduction and main results}
This paper concerns numerical integration for multivariate functions defined over the $d$-dimensional unit cube. For a Riemann integrable function $f: [0,1)^d\to \RR$, we approximate its integral
\[ I_d(f)=\int_{[0,1)^d}f(\bsx)\rd \bsx \]
by
\[ Q_{d,n}(f)=\sum_{h=0}^{n-1}w_h f(\bsx_h) \]
with sets of $n$ sampling points $\{\bsx_0,\ldots,\bsx_{n-1}\}\subset [0,1)^d$ and associated weights $\{w_0,\ldots,w_{n-1}\}.$ Quasi-Monte Carlo (QMC) rule denotes a special case of $Q_{d,n}$ where all the weights $w_h$ are equal to $1/n$. The worst-case error of an algorithm $Q_{d,N}$ in a Banach space $F$ with norm $\|\cdot\|$ is defined by
\[ e^{\wor}(F,Q_{d,n}):=\sup_{f\in F, \|f\|\leq 1}\left| I_d(f)-Q_{d,n}(f)\right|. \] 
In the field of information-based complexity \cite{NW08, NW10, TWW88}, we are interested in how the \emph{information complexity} $n(\varepsilon, d, F)$ grows in the reciprocal of the error tolerance $\varepsilon\in (0,1)$ and the dimension $d$. Here, the information complexity is defined as the minimum number of function values, among all possible $Q_{d,n}$, needed to make the worst-case error in $F$ no greater than $\varepsilon$, that is,
\[ n(\varepsilon, d, F):= \min\{ n\in \NN\, \mid \, \exists Q_{d,n}: e^{\wor}(F,Q_{d,n})\leq \varepsilon\}.\]

In a recent work by the author \cite{G23}, it has been proven that the information complexity for the following unweighted subspace of the Wiener algebra grows only polynomially both in $\varepsilon^{-1}$ and $d$:
\[ F_d^1:=\left\{ f\in C([0,1)^d)\:\middle| \: \|f\|:=\sum_{\bsk\in \ZZ^{d}}|\hat{f}(\bsk)|\max\left(1,\min_{j\in \supp(\bsk)}\log |k_j|\right)<\infty \right\},\]
with $\hat{f}(\bsk)$ being the $\bsk$-th Fourier coefficient of $f$, i.e.,
\[ \hat{f}(\bsk) = \int_{[0,1)^d}f(\bsx)\exp(-2\pi i\bsk\cdot \bsx)\rd \bsx,\]
and $\supp(\bsk):=\{j\in \{1,\ldots,d\}\mid k_j\neq 0\}$. More precisely, it has been shown that an upper bound $n(\varepsilon, d, F_d^1)\leq C_1 \varepsilon^{-3}d^3$ holds for a positive constant $C_1$, concluding that the problem of multivariate integration in $F_d^1$ is \emph{polynomially tractable.} We refer to \cite{K23, KV23} for more recent progress on this line of research.  In this context, an \emph{unweighted} function space $F$ refers to a space where all variables and groups of variables play an equal role. Therefore, for any permutation matrix $\pi$ and $f\in F$, it holds that $f\circ \pi \in F$ and $\|f\circ \pi\|=\|f\|$. The result presented in \cite{G23} builds upon the work of Dick \cite{D14}, who established polynomial tractability for multivariate integration in the intersection of the Wiener algebra and an unweighted space of H\"{o}lder continuous functions.

As a continuation of \cite{G23}, we prove the following result in this paper:
\begin{theorem}\label{thm:main1}
    Let $F_d^2$ be a subspace of the Wiener algebra defined by
    \[ F^2_d:=\left\{ f\in C([0,1)^d)\:\middle| \: \|f\|:=\sum_{\bsk\in \ZZ^{d}}|\hat{f}(\bsk)|\max\left(\width(\supp(\bsk)),\min_{j\in \supp(\bsk)}\log |k_j|\right)<\infty \right\}, \]
    where $\width: 2^{\{1,\ldots,d\}}\to \{1,\ldots,d\}$ is defined by
    \[ \width(u):=\max_{j\in u}j-\min_{j\in u}j+1,\]
    for non-empty subset $u\subseteq \{1,\ldots,d\}$, and $\width(\emptyset)=1.$ Then, there exists a positive constant $C_2$ such that, for any $d\in \NN$ and $\varepsilon \in (0,1)$, we have
    \[ n(\varepsilon, d, F_d^2)\leq C_2 \varepsilon^{-3}/(\log \varepsilon^{-1}). \]
\end{theorem}

In comparison to the result of \cite{G23} for $F_d^1$, by replacing 1 (the first argument in taking the maximum for each $\bsk$) with $\width(\supp(\bsk))$ in the definition of norms, the polynomial dependence of the information complexity on the dimension $d$ does not show up anymore, meaning that the problem of multivariate integration in $F_d^2$ is \emph{strongly polynomially tractable.} This result is strengthened by the following theorem on the former space $F_d^1$.

\begin{theorem}\label{thm:main2}
    For any linear algorithm $Q_{d,n}$ using $n$ function values, we have $e^{\wor}(F_d^1,Q_{d,n})\geq d/(2n^2)$ for any $d\in \NN$ and $n>2d$.
\end{theorem}

Note that there is a significant gap between the lower bound on the worst-case error obtained above and the upper bound of order $dn^{-1/3}$ shown in \cite{G23}. Nevertheless, this result implies that a dependence of the information complexity on the dimension $d$ cannot be eliminated for $F_d^1$. Therefore, the problem of multivariate integration in $F_d^1$ is polynomially tractable but \emph{not} strongly polynomially tractable. As a future research direction, it would be interesting to study whether an intermediate space between $F_d^1$ and $F_d^2$ still exhibits strong polynomial tractability for multivariate integration. As we have $1\leq |\supp(\bsk)|\leq \width(\supp(\bsk))$ for all $\bsk\in \ZZ^d$ when defining $|\supp(\bszero)|=1$, one of the most natural spaces we can consider is an unweighted space
\[ F_d^3:=\left\{ f\in C([0,1)^d)\:\middle| \: \|f\|:=\sum_{\bsk\in \ZZ^{d}}|\hat{f}(\bsk)|\max\left(|\supp(\bsk)|,\min_{j\in \supp(\bsk)}\log |k_j|\right)<\infty \right\}. \]

Note that, although the space $F_d^2$ is weighted, it remains invariant under the reversion of the variables, i.e., if $f\in F_d^2$, then we have $g\in F_d^2$ and $\|f\|=\|g\|$ where $g(x_1,\ldots,x_d)=f(x_d,\ldots,x_1)$. This is in contrast to many existing results on strong polynomial tractability for multivariate integration in the worst-case setting, where weight parameters are introduced to model the relative importance of each group of variables, and variables are typically assumed ordered in decreasing importance order. See \cite{DGPW17,DP15,NW08,NW10,SW98} among many others. In fact, it seems not possible to characterize the space $F_d^2$ in such a way. The author believes that further tractability studies in subspaces of the Wiener algebra will offer new insights into the field of information-based complexity, particularly regarding (strong) polynomial tractability in (un)weighted spaces.

%%%%%%%%%%%%%%%%%%%%%%%%%%%%%%
%%%%%%%%%%%%%%%%%%%%%%%%%%%%%%
\section{Proof of Theorem~\ref{thm:main1}}
This section is devoted to proving Theorem~\ref{thm:main1} by providing an explicit QMC rule that attains the desired worst-case error bound. The QMC rule considered here is exactly the same as the one discussed in \cite{G23}. For an integer $m\geq 2$, let
\[ \PP_m := \{\lceil m/2\rceil<p\leq m\, \mid\, \text{$p$ is prime} \}.\]
It is known that there exist constants $c_{\PP}$ and $C_{\PP}$ with $0<c_{\PP}<\min(1,C_{\PP})$ such that
\begin{align}\label{eq:prime_num}
c_{\PP}\frac{m}{\log m}\leq |\PP_m|\leq C_{\PP}\frac{m}{\log m}, \end{align}
for all $m\geq 2$, see \cite[Corollaries~1--3]{RS62}. Now, given an integer $m\geq 2$, we define two different point sets as multiset unions:
\[ P_{d,m}^1=\bigcup_{p\in \PP_m}S_{d,p}\quad \text{and}\quad P_{d,m}^2=\bigcup_{p\in \PP_m}T_{d,p},\]
where $S_{d,p}=\{\bsx_h^{(p)}\mid 0\leq h<p^2\}$ and $T_{d,p}=\{\bsx_{h,\ell}^{(p)}\mid 0\leq h,\ell<p\}$ are sets with $p^2$ points known as \emph{Korobov's $p$-sets} \cite{D14,DP15,HW81,K63}. These point sets are defined as follows:
\[ \bsx_h^{(p)}=\left( \left\{ \frac{h}{p^2}\right\}, \left\{ \frac{h^2}{p^2}\right\},\ldots, \left\{ \frac{h^d}{p^2}\right\}\right),\]
and
\[  \bsx_{h,\ell}^{(p)}=\left( \left\{ \frac{h\ell}{p}\right\}, \left\{ \frac{h\ell^2}{p}\right\},\ldots, \left\{ \frac{h\ell^d}{p}\right\}\right), \]
respectively, where we write $\{x\}=x-\lfloor x\rfloor$ to denote the fractional part of a non-negative real number $x$. It is important to note that taking the multiset unions of Korobov's $p$-sets with different primes $p$ is crucial in our error analysis. Trivially we have
\[ |P_{d,m}^1|=|P_{d,m}^2|=\sum_{p\in \PP_m}p^2.\]

The following result on the exponential sums refines the known results from \cite[Lemmas~4.5 \& 4.6]{HW81} as well as \cite[Lemmas~4.4 \& 4.5]{DP15}.
\begin{lemma}\label{lem:exponential_sum}
    Let $d\in \NN$ and $p$ be a prime with $p\geq d$. For any $\bsk\in \ZZ^d\setminus \{\bszero\}$ such that there exists at least one index $j^*\in \{1,\ldots,d\}$ where $k_{j^*}$ is not divisible by $p$, i.e., $p\nmid \bsk$, the following bounds hold:
    \[ \left|\frac{1}{p^2}\sum_{h=0}^{p^2-1}\exp\left( 2\pi i\bsk\cdot \bsx_h^{(p)}\right)\right| \leq \frac{\width(\supp(\bsk))}{p},\]
    and
    \[ \left|\frac{1}{p^2}\sum_{h,\ell=0}^{p-1}\exp\left( 2\pi i\bsk\cdot \bsx_{h,\ell}^{(p)}\right)\right| \leq \frac{\width(\supp(\bsk))}{p}.\]
\end{lemma}

\begin{proof}
    Let us consider the first bound. As we have $\{0,\ldots,p^2-1\}=\{h_0+h_1p \mid 0\leq h_0,h_1<p\}$ and, for each pair of $h_0,h_1\in \{0,\ldots,p-1\}$, it holds that
    \begin{align*}
        \exp\left( 2\pi i\bsk\cdot \bsx_{h_0+h_1p}^{(p)}\right) & = \exp\left( \frac{2\pi i}{p^2}\sum_{j\in \supp(\bsk)}k_j (h_0+h_1p)^j\right) \\
        & = \exp\left( \frac{2\pi i}{p^2}\sum_{j\in \supp(\bsk)}k_j \sum_{a=0}^{j}\binom{j}{a}h_0^a(h_1p)^{j-a}\right) \\
        & = \exp\left( \frac{2\pi i}{p^2}\sum_{j\in \supp(\bsk)}k_j (h_0^j+j h_0^{j-1}h_1p)\right),
    \end{align*}
    we obtain
    \begin{align*}
        \left|\frac{1}{p^2}\sum_{h=0}^{p^2-1}\exp\left( 2\pi i\bsk\cdot \bsx_h^{(p)}\right)\right| & = \left|\frac{1}{p^2}\sum_{h_0,h_1=0}^{p-1}\exp\left( \frac{2\pi i}{p^2}\sum_{j\in \supp(\bsk)}k_j (h_0^j+j h_0^{j-1}h_1p)\right)\right| \\
        & = \left|\frac{1}{p}\sum_{h_0=0}^{p-1}\exp\left( \frac{2\pi i}{p^2}\sum_{j\in \supp(\bsk)}k_j h_0^j\right)\frac{1}{p}\sum_{h_1=0}^{p-1}\exp\left(\frac{2\pi i h_1}{p}\sum_{j\in \supp(\bsk)}k_j j h_0^{j-1}\right)\right| \\
        & \leq \frac{1}{p}\sum_{h_0=0}^{p-1}\left|\frac{1}{p}\sum_{h_1=0}^{p-1}\exp\left(\frac{2\pi i h_1}{p}\sum_{j\in \supp(\bsk)}k_j j h_0^{j-1}\right)\right|\\
        & = \frac{1}{p}\sum_{\substack{h_0=0\\ \sum_{j\in \supp(\bsk)}k_j j h_0^{j-1}\equiv 0 \pmod p}}^{p-1}1,
    \end{align*}
    where the last equality follows from the well-known character property for the trigonometric functions \cite[Lemma~4.3]{DHP15}. Here, by denoting $j_{\min}=\min_{j\in \supp(\bsk)}j$ and $j_{\max}=\max_{j\in \supp(\bsk)}j$, we have
    \[ \sum_{j\in \supp(\bsk)}k_j j h_0^{j-1}=\sum_{\substack{j=j_{\min}\\ j\in \supp(\bsk)}}^{j_{\max}}k_j j h_0^{j-1}= h_0^{j_{\min}-1}\sum_{\substack{j=j_{\min}\\ j\in \supp(\bsk)}}^{j_{\max}}k_j j h_0^{j-j_{\min}}.\]
    As the last sum over $j$ is a polynomial in $h_0$ with degree $j_{\max}-j_{\min}$, the number of solutions of the congruence $\sum_{j\in \supp(\bsk)}k_j j h_0^{j-1}\equiv 0 \pmod p$ is at most $j_{\max}-j_{\min}+1=\width(\supp(\bsk))$. Thus the result follows. Since the second bound can be proven in the same manner, we omit the details.
\end{proof}

Note that, if $k_j$ is divisible by $p$ for all $j$, i.e., $p\mid \bsk$, then we only have a trivial bound on the exponential sum, which is 1. Using this refined result, we obtain the following bounds on the exponential sums for our point sets $P_{d,m}^1$ and $P_{d,m}^2$.
\begin{corollary}\label{cor:exponential_sum}
    Let $d\in \NN$ and $m\geq 2$ with $\min_{p\in \PP_m}p\geq d$. For any $\bsk\in \ZZ^d\setminus \{\bszero\}$, it holds that
    \[ \left|\frac{1}{|P_{d,m}^1|}\sum_{p\in \PP_m}\sum_{h=0}^{p^2-1}\exp\left( 2\pi i\bsk\cdot \bsx_h^{(p)}\right)\right| \leq \frac{1}{m}\left( 4\width(\supp(\bsk))+\frac{8}{c_{\PP}}\min_{j\in \supp(\bsk)}\log |k_j|\right),\]
    and
    \[ \left|\frac{1}{|P_{d,m}^2|}\sum_{p\in \PP_m}\sum_{h,\ell=0}^{p-1}\exp\left( 2\pi i\bsk\cdot \bsx_{h,\ell}^{(p)}\right)\right| \leq \frac{1}{m}\left( 4\width(\supp(\bsk))+\frac{8}{c_{\PP}}\min_{j\in \supp(\bsk)}\log |k_j|\right).\]
\end{corollary}

\begin{proof}
    The following proof for the first bound is similar to that of \cite[Corollary~2.3]{G23}, and the second bound can be proven in a similar way, so we omit the details. Using Lemma~\ref{lem:exponential_sum}, we have
    \begin{align*}
        \left|\frac{1}{|P_{d,m}^1|}\sum_{p\in \PP_m}\sum_{h=0}^{p^2-1}\exp\left( 2\pi i\bsk\cdot \bsx_h^{(p)}\right)\right| & \leq \frac{1}{|P_{d,m}^1|}\sum_{p\in \PP_m}\left|\sum_{h=0}^{p^2-1}\exp\left( 2\pi i\bsk\cdot \bsx_h^{(p)}\right)\right| \\
        & \leq \frac{1}{|P_{d,m}^1|}\sum_{\substack{p\in \PP_m\\ p\nmid \bsk}}p\width(\supp(\bsk))+\frac{1}{|P_{d,m}^1|}\sum_{\substack{p\in \PP_m\\ p\mid \bsk}}p^2\\
        & \leq \frac{m|\PP_m|}{|P_{d,m}^1|}\width(\supp(\bsk))+\frac{m^2}{|P_{d,m}^1|}\sum_{\substack{p\in \PP_m\\ p\mid \bsk}}1\\
        & \leq \frac{m|\PP_m|}{(m/2)^2|\PP_m|}\width(\supp(\bsk))+\frac{m^2}{(m/2)^2|\PP_m|}\sum_{\substack{p\in \PP_m\\ p\mid \bsk}}1\\
        & \leq \frac{4}{m}\width(\supp(\bsk))+\frac{4\log m}{c_{\PP}m}\sum_{\substack{p\in \PP_m\\ p\mid \bsk}}1,
    \end{align*}
    where the last inequality follows from \eqref{eq:prime_num}. To give a bound on the last sum over $p\in \PP_m$ which divides $\bsk$, we use the fact that, for any integers $k,n\in \NN$, $k$ has at most $\log_n k$ prime divisors larger than or equal to $n$. With $\II(\cdot)$ denoting the indicator function, for any index $j^*\in \supp(\bsk)$, we get
    \begin{align*}
        \sum_{\substack{p\in \PP_m\\ p\mid \bsk}}1 & = \sum_{p\in \PP_m}\prod_{j\in \supp(\bsk)}\II(p\mid k_j)\leq \sum_{p\in \PP_m}\II(p\mid k_{j^*}) \leq \log_{\lceil m/2\rceil +1}|k_{j^*}|\leq \frac{2\log |k_{j^*}|}{\log m}.
    \end{align*}
    Since this inequality applies to any index $j^*\in \supp(\bsk)$, it holds that
    \[ \sum_{\substack{p\in \PP_m\\ p\mid \bsk}}1 \leq \frac{2}{\log m}\min_{j\in \supp(\bsk)}\log |k_j|.\]
    This completes the proof.
\end{proof}

Now we are ready to prove Theorem~\ref{thm:main1}.
\begin{proof}[Proof of Theorem~\ref{thm:main1}]
    Since any function $f\in F_d^2$ has an absolutely convergent Fourier series, by letting $Q_{d,n}$ being the QMC rule using $P_{d,m}^1$ (or $P_{d,m}^2$) for some $m\geq 2$ with $\min_{p\in \PP_m}p\geq d$, it follows from Corollary~\ref{cor:exponential_sum} that, with $n$ equal to $\sum_{p\in \PP_m}p^2$,
    \begin{align*}
        \left| I_d(f)-Q_{d,n}(f)\right| & = \left| I_d(f)-\frac{1}{|P_{d,m}^1|}\sum_{p\in \PP_m}\sum_{h=0}^{p^2-1}f(\bsx_h^{(p)})\right| \\
        & = \left| \hat{f}(\bszero)-\frac{1}{|P_{d,m}^1|}\sum_{p\in \PP_m}\sum_{h=0}^{p^2-1}\sum_{\bsk\in \ZZ^d}\hat{f}(\bsk)\exp\left( 2\pi i\bsk\cdot \bsx_h^{(p)}\right)\right| \\
        & = \left| \sum_{\bsk\in \ZZ^d\setminus \{\bszero\}}\hat{f}(\bsk)\frac{1}{|P_{d,m}^1|}\sum_{p\in \PP_m}\sum_{h=0}^{p^2-1}\exp\left( 2\pi i\bsk\cdot \bsx_h^{(p)}\right)\right|\\
        & \leq \sum_{\bsk\in \ZZ^d\setminus \{\bszero\}}|\hat{f}(\bsk)|\left| \frac{1}{|P_{d,m}^1|}\sum_{p\in \PP_m}\sum_{h=0}^{p^2-1}\exp\left( 2\pi i\bsk\cdot \bsx_h^{(p)}\right)\right|\\
        & \leq \frac{1}{m}\sum_{\bsk\in \ZZ^d\setminus \{\bszero\}}|\hat{f}(\bsk)|\left( 4\width(\supp(\bsk))+\frac{8}{c_{\PP}}\min_{j\in \supp(\bsk)}\log |k_j|\right) \\
        & \leq \frac{16}{c_{\PP}m}\sum_{\bsk\in \ZZ^d\setminus \{\bszero\}}|\hat{f}(\bsk)|\max\left(\width(\supp(\bsk)),\min_{j\in \supp(\bsk)}\log |k_j|\right) \leq \frac{16}{c_{\PP}m}\|f\|.
    \end{align*}
    This leads to an upper bound on the worst-case error as
    \[ e^{\wor}(F_d^2,Q_{d,n})\leq \frac{16}{c_{\PP}m}.\]
    Therefore, in order to make $e^{\wor}(F_d^2,Q_{d,n})$ less than or equal to $\varepsilon \in (0,1)$, it suffices to choose $m=\lceil 16c_{\PP}^{-1}\varepsilon^{-1}\rceil$ and we have
    \[  n(\varepsilon, d, F_d^2)\leq \sum_{p\in \PP_{\lceil 16c_{\PP}^{-1}\varepsilon^{-1}\rceil}}p^2 \leq C_{\PP}\frac{\lceil 16c_{\PP}^{-1}\varepsilon^{-1}\rceil}{\log \lceil 16c_{\PP}^{-1}\varepsilon^{-1}\rceil}\times \left(\lceil 16c_{\PP}^{-1}\varepsilon^{-1}\rceil\right)^2, \]
    from which the result follows immediately.
\end{proof}

%%%%%%%%%%%%%%%%%%%%%%%%%%%%%%
%%%%%%%%%%%%%%%%%%%%%%%%%%%%%%
\section{Proof of Theorem~\ref{thm:main2}}

\begin{proof}[Proof of Theorem~\ref{thm:main2}]
We adopt a similar approach as in the proofs of \cite[Theorem~1]{SW97} and \cite[Theorem~1]{DLPW11}. Consider an arbitrary linear algorithm $Q_{d,n}(f)=\sum_{h=0}^{n-1}w_h f(\bsx_h)$. For a set $\Acal\subset \ZZ^d$ with enough cardinality $|\Acal|> n$, we define a function $g: [0,1)^d\to \CC$ by 
\[ g(\bsx)=\sum_{\bsk\in \Acal}c_{\bsk}\exp(2\pi i\bsk\cdot \bsx)\]
with $c_{\bsk}\in \CC$, which satisfies $g(\bsx_h)=0$ for all $h=0,\ldots,n-1$. In fact, there exists a non-zero vector of $(c_{\bsk})_{\bsk\in \Acal}$, as the condition that $g(\bsx_h)=0$ for all $h=0,\ldots,n-1$ forms $n$ homogeneous linear equations with $|\Acal|>n$ unknowns $c_{\bsk}$. Let us normalize these coefficients in such a way that
\[ \max_{\bsk\in \Acal}|c_{\bsk}|=c_{\bsell}=1\quad \text{for some $\bsell\in \Acal$.}\]

With this $\bsell$ and a positive constant $C$, we define another function $\tilde{g}: [0,1)^d\to \CC$ as follows:
\[ \tilde{g}(\bsx)=C\exp(-2\pi i\bsell\cdot \bsx)g(\bsx)=C\sum_{\bsk\in \Acal}c_{\bsk}\exp(2\pi i(\bsk-\bsell)\cdot \bsx). \]
Then we construct a real-valued function $g^{\star}$ defined on $[0,1)^d$ by taking the average of $\tilde{g}$ and its complex conjugate: $g^{\star}(\bsx)=(\tilde{g}(\bsx)+\overline{\tilde{g}(\bsx)})/2$. 
Regarding the norm of $g^{\star}$ in $F_d^1$, we have
\begin{align*}
    \|g^{\star}\| & \leq \frac{\|\tilde{g}\|+\|\overline{\tilde{g}}\|}{2} = \|\tilde{g}\| = C\sum_{\bsk\in \Acal}|c_{\bsk}|\max\left(1,\min_{j\in \supp(\bsk)}\log |k_j-\ell_j|\right)\\
    & \leq C\sum_{\bsk\in \Acal}\max\left(1,\min_{j\in \supp(\bsk)}\log |k_j-\ell_j|\right) \leq C\max_{\bsell\in \Acal}\sum_{\bsk\in \Acal}\max\left(1,\min_{j\in \supp(\bsk)}\log |k_j-\ell_j|\right).
\end{align*}
To ensure $\|g^{\star}\|\leq 1$, we set
\[ C=\left( \max_{\bsell\in \Acal}\sum_{\bsk\in \Acal}\max\left(1,\min_{j\in \supp(\bsk)}\log |k_j-\ell_j|\right)\right)^{-1}.\]

By construction, we have $g^{\star}(\bsx_h)=0$ for all $h=0,\ldots,n-1$, which implies $Q_{n,d}(g^{\star})=0$. On the other hand, the exact integral is given by
\[ I_d(g^{\star})=Cc_{\bsell}=C=\left( \max_{\bsell\in \Acal}\sum_{\bsk\in \Acal}\max\left(1,\min_{j\in \supp(\bsk)}\log |k_j-\ell_j|\right)\right)^{-1}.\]
Since $g^{\star}\in F_d^1$ with $\|g^{\star}\|\leq 1$, the worst-case error of any linear algorithm $Q_{d,n}$ is bounded below by
\[ e^{\wor}(F_d^1,Q_{d,n})\geq \left| I_d(g^{\star})-Q_{n,d}(g^{\star})\right|=\left( \max_{\bsell\in \Acal}\sum_{\bsk\in \Acal}\max\left(1,\min_{j\in \supp(\bsk)}\log |k_j-\ell_j|\right)\right)^{-1}.\]

In what follows, let 
\[ \Acal=\{\bszero\} \cup \left\{ \bsk\in \ZZ^d\, \mid\,  \text{$d-1$ of $k_j$ are all $0$ and one non-zero $k_j$ is from $\{1,\ldots,\lceil n/d\rceil\}$}\right\}. \]
It is easy to verify that $|\Acal|=1+d\lceil n/d\rceil > n$. For this choice of $\Acal$, we can restrict ourselves to $\bsell =(\ell,0,\ldots,0)$ for some $\ell\in \{0,\ldots,\lceil n/d\rceil\}$. By utilizing the assumption $n>2d$ and the well-known inequality $\log x \leq x-1$, we have
\begin{align*}
    & \sum_{\bsk\in \Acal}\max\left(1,\min_{j\in \supp(\bsk)}\log |k_j-\ell_j|\right) \\
    & = \max\left(1,\log \ell\right) + \sum_{k_1=1}^{\lceil n/d\rceil}\max\left(1,\log |k_1-\ell|\right) + \sum_{j=2}^{d}\sum_{k_j=1}^{\lceil n/d\rceil}\max\left(1,\log k_j\right) \\
    & \leq \log \lceil n/d\rceil + d\sum_{k=1}^{\lceil n/d\rceil}\max\left(1,\log k\right) \\
    & \leq \log \lceil n/d\rceil + d\lceil n/d\rceil\log \lceil n/d\rceil\\
    & \leq \left(\lceil n/d\rceil-1\right)\cdot \left( 1+d\lceil n/d\rceil\right)\leq \frac{2n^2}{d}.
\end{align*}
Since the last bound is independent of $\ell$, we obtain
\[ e^{\wor}(F_d^1,Q_{d,n})\geq \left( \max_{\bsell\in \Acal}\sum_{\bsk\in \Acal}\max\left(1,\min_{j\in \supp(\bsk)}\log |k_j-\ell_j|\right)\right)^{-1}\geq \frac{d}{2n^2}.\]
This completes the proof.
\end{proof}

\bibliographystyle{plain}
\bibliography{ref.bib}

\end{document}